\documentclass[11pt]{amsart}

\usepackage{graphicx}

\usepackage{amssymb}
\usepackage{amsmath}
\usepackage{mathrsfs}
\usepackage{mathtools}
\usepackage[hyphens]{url} 
\usepackage[margin=1.4in]{geometry}
\usepackage{hyperref}
\usepackage{xcolor}

\newcommand{\overbar}[1]{\mkern 1.5mu\overline{\mkern-1.5mu#1\mkern-1.5mu}\mkern 1.5mu}

\usepackage{a4wide,amsfonts, latexsym, euscript,eufrak, units,mathrsfs,bbold}   

\def\HH{\EuFrak H}
\def\R{\mathbb R}
\def\E{\mathbb E}

\def\dom{{\rm Dom}}

\newtheorem{prop}{Proposition}[section]

\newtheorem{lemma}[prop]{Lemma}

\newtheorem{theorem}[prop]{Theorem}

\newtheorem{remark}[prop]{Remark}

\numberwithin{equation}{section}
\newcommand{\norm}[1]{\left\lVert#1\right\rVert}

\usepackage [english]{babel}
\usepackage [autostyle, english = american]{csquotes}
\MakeOuterQuote{"}

\begin{document}

\title{Continuous Breuer-Major Theorem for Vector Valued Fields}

\author{David Nualart}
\address{University of Kansas, U.S.A.}
\email{nualart@ku.edu}

 \author{Abhishek Tilva}
 \address{Indian Institute of Science Education and Research, Pune, India}
 \email{aktilva1@gmail.com}
  \thanks{The first author was supported by the NSF grant DMS 1811181.}
 \thanks{The second author was supported by DST-INSPIRE scholarship.}





\begin{abstract}

Let $\xi : \Omega \times \mathbb{R}^n \to \mathbb{R}$ be zero mean, mean-square continuous, stationary,  Gaussian random field with covariance function $r(x) = \mathbb{E}[\xi(0)\xi(x)]$ and let $G : \mathbb{R} \to \mathbb{R}$ such that $G$ is square integrable with respect to the standard Gaussian measure and is of Hermite rank $d$. The Breuer-Major theorem in it's continuous setting gives that, if $r \in L^d(\mathbb{R}^n)$, then the finite dimensional distributions of $Z_s(t) = \frac{1}{(2s)^{n/2}} \int_{[-st^{1/n},st^{1/n}]^n} \Big[G(\xi(x)) - \mathbb{E}[G(\xi(x))]\Big]dx$ converge to that of a scaled Brownian motion as $s \to \infty$. Here we give a proof for the case when $\xi : \Omega \times \mathbb{R}^n \to \mathbb{R}^m$ is a random vector field. We also give a proof for the functional convergence in $C([0,\infty))$ of $Z_s$ to hold under the condition that for some $p>2$, $G\in L^p(\R^m, \gamma_m)$ where $\gamma_m$ denotes the standard Gaussian measure on $\R^m$ and we derive expressions for the asymptotic variance of the second chaos component in the Wiener chaos expansion of   $Z_s(1)$.
\end{abstract}

\maketitle

\section{Introduction}

The classical Breuer-Major theorem in its primitive form, as proved first by P\'eter Breuer and P\'eter Major in their seminal paper \cite{breuer_central_1983} in 1983, states that, under an appropriate condition involving the covariances, the sum of a functional of a stationary sequence of Gaussian variables, scaled by  the square root of the number of terms, converges in distribution to a Gaussian variable.  A formal statement is as follows.
For a centered stationary sequence of Gaussian variables $\{\xi_k : k \in \mathbb{Z}\}$ with unit variance and a function $G \in L^2(\mathbb{R}, \gamma_1)$ of Hermite rank $d$, where $\gamma_1$ denotes the standard Gaussain measure on $\mathbb{R}$,  if $\sum_{k \in \mathbb{Z}} |\mathbb{E}[\xi_1\xi_{1+k}]|^d < \infty$, then the following convergence in law holds
\[
\frac{1}{\sqrt{n}} \Big[\sum_{k=1}^n G(\xi_k) - n\mathbb{E}[G(\xi_1)]\Big] \Rightarrow \mathcal{N}(0,V)\] 
as $n \to \infty$, for some $V \in [0, \infty)$.

The theorem has now become one of the most celebrated and widely applicable results in stochastic analysis. An extension of the original version to sequences of vectors was done by Arcones in \cite{arcones_limit_1994} and continuous versions of the theorem for real valued fields are found in \cite{darses_limit_2010,hariz_limit_2002,ivanov_statistical_1989}.

A continuous version of this theorem  (see Theorem 2.3.1 of \cite{ivanov_statistical_1989})  asserts that for a zero mean, stationary, isotropic Gaussian random field $\xi : \Omega \times \mathbb{R}^n \to \mathbb{R}$ with covariance function $r(x) = \mathbb{E}[\xi(0)\xi(x)]$, if $r \in L^d(\mathbb{R}^n)$ and $r(x) \to 0$ as $|x| \to \infty$, then as $s \to \infty$, the finite dimensional distributions of the processes
\[ 
Z_s(t) = \frac{1}{s^{n/2}} \int_{B_n(st^{1/n})} \Big[G(\xi(x)) - \mathbb{E}[G(\xi(x))]\Big]dx, \; \; t \in [0, \infty)
\]
converge to those of a scaled Brownian motion. Here $B_n(a)$ denotes the  ball of radius $a$ centered at the origin in $\mathbb{R}^n$.

Estrade and Le\'on in \cite{estrade_central_2016} have partially addressed the case of random vector fields on  the Euclidean space when they mention adapting the Breuer-Major theorem to prove a Central Limit Theorem for  the Euler characteristic of an excursion set (see Proposition 2.4 of \cite{estrade_central_2016}).

The purpose of this paper is to  obtain a mutidimensional extension of the  continuous Breuer-Major theorem for random fields, including the corresponding invariance principle.  We will use the $n$-cubes $[-s,s]^n$ instead of balls as expanding sets and we prove it without the assumption of isotropy. We will also give a proof for the convergence of $Z_s$ to hold in a functional sense, i.e. convergence in law in $C([0, \infty))$ under the condition that $G \in L^p(\R^m, \gamma_m)$ for some $p > 2$, where $\gamma_m$ denotes the standard normal distribution on $\R^m$.  This remains an unaddressed question in the literature in the case of vectors. The approach here is similar to the method that has been employed in \cite{campese_continuous_2018} and \cite{nourdin_functional_2018}, namely  using the  representation by means of the Malliavin divergence operator, which is obtained through a  shift operator, and applying Meyer inequalities to show tightness.  However,  in the case of vectors fields,  this approach is  more involved   and  requires the introduction of weighted shift operators.

The modern proof of the Breuer-Major theorem is based on the Stein-Malliavin approach and is presented in \cite{nourdin_normal_2012}. We will rely on this methodology for the proofs. We refer the reader to  the monographs \cite{nourdin_normal_2012} or \cite{nualart_malliavin_2006} for unexplained usage of terms.
 
The organization of the paper is as follows. Section 2 describes the necessary framework and notations.  The third section contains the statements of our results. In  Section 4  we briefly describe several preliminary results and definitions regarding Malliavin calculus on Wiener space and we write the Wiener chaos expansions of variables of interest. Finally, Section 5 contains the proofs.

\section{Setup}

Let $\xi_i : \Omega \times \mathbb{R}^n \to \mathbb{R}$, $i = 1,\dots,m$ be zero mean, mean-square continuous, stationary Gaussian random fields which are \emph{jointly stationary}, i.e., for $1 \leq i,j \leq m$, the cross covariance functions, $r_{i,j}(x,y) = \mathbb{E}[(\xi_i(x)\xi_j(y)] = r_{i,j}(x-y)$ (in an abuse of notation), depend only on $x-y$. Then the function $r : \mathbb{R}^n \to M_m(\mathbb{R})$, $r(x) = (r_{i,j}(x))_{1\leq i,j \leq m}$ is the covariance function for the vector valued field, 
\[
\xi : \Omega \times \mathbb{R}^n \to \mathbb{R}^m, \; \; \xi(x) = (\xi_i(x))_{1\leq i \leq m}.
\]

We now recall the Hermite polynomials in the multivariate case. We denote the $n$-th Hermite polynomial by 
\begin{equation} \label{hermite}
H_n(x) = (-1)^n \; e^{(x^2/2)}\frac{d^n}{dx^n}e^{(-x^2/2)}.
\end{equation}
For any multi-index $a = (a_1,\dots,a_m), a_i \in \mathbb{N} \cup \{0\}$, we write $|a| = \sum_{i=1}^m a_i$, $a! = \prod_{i=1}^m a_i!$ and
\begin{equation} \label{hermite2}
\overbar{H}_a(x) = \prod_{i=i}^m H_{a_i}(x_i).
\end{equation}
We then have that $\{\frac{1}{\sqrt{a!}}\overbar{H}_a : a \; \text{is a multi-index}\}$ is an orthonormal basis of $L^2(\mathbb{R}^m, \gamma_m)$, where $\gamma_m$ denotes the standard Gaussian measure on $\mathbb{R}^m$ (see \cite{bogachev_gaussian_1998}). 

Let $G : \mathbb{R}^m \to \mathbb{R}$ be such that $G$ is not a constant and $G \in L^2(\mathbb{R}^m, \gamma_m)$. Denoting, $\mathcal{I}_q = \{a \in \mathbb{Z}^m : a_i \geq 0, |a|=q\}$, we have the following expansion of $G$ where the convergence of the series is in $L^2$ sense,
\begin{equation}
\sum_{q=0}^{\infty} \sum_{a \in \mathcal{I}_q} c(G,a) \overbar{H}_a(x) = G(x).
\end{equation}
In above expansion $c(G,a) = \frac{1}{a!}\int_{\mathbb{R}^m} G(x) \overbar{H}_a(x) \gamma_m(dx)$. Let $G_0 = \int_{\mathbb{R}^m} G(x) \gamma_m(dx) = 0$ and call an integer $d \in \mathbb{N}$ to be the \emph{rank} of $G$ if there exist a multi-index $a$ such that $|a| = d$ and $c(G,a) \neq 0$. Therefore, 
\begin{equation} \label{Gexp}
\sum_{q=d}^{\infty} \sum_{a \in \mathcal{I}_q} c(G,a) \overbar{H}_a(x) = G(x).
\end{equation}
For any integer $q\ge 1$, we will make use of the notation
\begin{equation}  \label{Gq}
G_q(x) = \sum_{a \in \mathcal{I}_q} c(G,a) \overbar{H}_a(x).
\end{equation}

We are interested in the asymptotic behavior as $s\rightarrow  \infty$ of the random variables defined by
\begin{equation}  \label{Ls}
L_s = \frac{1}{(2s)^{n/2}}\int_{[-s,s]^n}G(\xi(x))dx.
\end{equation}
For any integer $q\ge 1$, we put
\begin{equation}  \label{Lsq}
L_s^{(q)} = \frac{1}{(2s)^{n/2}}\int_{[-s,s]^n}G_q(\xi(x))dx.
\end{equation}
Also we denote the variances  of $L_s$ and $L^{(q)}_s$  by Var$(L_s) = V_s$ and  Var$(L_s^{(q)}) = V_s^{(q)}$, respectively.
Set
\[
 C_G(x,y) = \mathbb{E}[G(\xi(x))G(\xi(y))]
 \]
  as the covariance function of $G(\xi(x))$. We ignore the degenerate case when $V_s = 0$ for all $s > 0$.

\begin{remark}
We will use Fubini-Tonelli's theorem to exchange integrals and expectation and everytime its use will be justified by Theorem 1.1.1 of \cite{ivanov_statistical_1989}. We will also use it to interchange the multiple Wiener-It\^o integral and Lebesgue integral.
\end{remark}

We will impose the following condition on the covariances. As noted in the proof of Theorem 1 of \cite{arcones_limit_1994}, given that $r(0)$ is invertible, by a linear transformation we can assume that $r(0) = $Id$_{m \times m}$ ($m \times m$ identity matrix).  Moreover,  recall that $d\ge 1$ is the Hermite rank of our functional $G$.

\hfill

\textbf{Condition (C1).} $r(0) = $Id$_{m \times m}$ and for every $1 \leq j,k \leq m$,   $r_{j,k} \in L^d(\mathbb{R}^n)$.

\begin{remark}
Since by Cauchy-Schwarz inequality and stationarity, $\mathbb{E}[\xi_j(x)\xi_k(0)] \leq 1$, \textbf{(C1)} implies that $r_{j,k} \in L^p(\mathbb{R}^n)$ for all $p \geq d$.
\end{remark}

\hfill

\section{Statements}

We are now in a position to state the main results of this paper. The lemma below provides a simple characterization for the asymptotic variance of $L_s$ defined in (\ref{Ls}). Note here that we have assumed $\mathbb{E}[G(\xi(0))] = 0$, that means the Hermite rank of $G$ is $d\ge 1$.

\begin{lemma}
Under \textbf{(C1)}, the random field $G \circ \xi : \Omega \times \mathbb{R}^n \to \mathbb{R}$ is weakly stationary, i.e. $C_G(x,y) = \mathbb{E}[G(\xi(x))G(\xi(y))] = C_G(x-y)$ is a function of $x-y$ and $C_G \in L^1(\mathbb{R}^n)$. The following also holds,
\begin{equation} \label{ecu1a}
 V := \underset{s \to \infty}{\lim} V_s = \int_{\mathbb{R}^n}C_G(x)dx < \infty,
\end{equation}
where we recall that $V_s$ denoted  the variance of the random variable  $L_s$ defined in (\ref{Ls}). 
\end{lemma}

\begin{theorem}
Under \textbf{(C1)}, 
\[
L_s = \frac{1}{(2s)^{n/2}} \int_{[-s,s]^n} G(\xi(x))dx \; \; \Rightarrow \; \; \mathcal{N}(0,V) \; \; as \; \; s \to \infty.
\]
Here $V$ is as in Lemma 3.1 and $\Rightarrow$ denotes convergence in law.

\end{theorem}

The above statement is a continuous version of Theorem 4 of \cite{arcones_limit_1994}.

\begin{theorem}
Under \textbf{(C1)} as $s \to \infty$, the finite dimensional distributions of the process
\[
Z_{s,y} = \frac{1}{(2s)^{n/2}} \int_{[-sy^{1/n},sy^{1/n}]^n} G(\xi(x)) dx, \; \; \; y \in [0,\infty),
\]
converge to those of  $\sqrt{V} B_y$ on $[0,\infty)$, where $B=\{B_y, y\ge 0\}$ is a standard Brownian motion. 
\end{theorem}

The above statement is a multi-dimensional extension of Theorem 2.3.1 of \cite{ivanov_statistical_1989}. The above two theorems are presented separately for better elucidation and to save on unnecessary notation. Clearly Theorem 3.3 contains Theorem 3.2.  

\begin{theorem}
Assume  \textbf{(C1)} and $G \in L^p(\mathbb{R}^m, \gamma_m)$ for some $p > 2$. As $s \to \infty$, the probability measures $\{P_s : s>0\}$ on $C([0, \infty))$ induced by $\{Z_s:s>0\}$ (as defined in Theorem 3.3) converge weakly to the probability measure induced by   $\sqrt{V}B_y$ on $C([0, \infty))$, where again $B$ denotes a standard Brownian motion.
\end{theorem}

The above  result is multi-dimensional counterpart of Theorem 1.1 of \cite{campese_continuous_2018}.
 
 \medskip
 Consider the  $m\times m$  symmetric matrix  $C = (c_{j,k})_{1\leq j,k \leq m} $ given by
\begin{equation}  \label{matrixC}
\begin{cases}
c_{j,k} &= \int_{\mathbb{R}^n}G(x)x_jx_k\phi_m(x)dx,    \quad   {\rm for} \,\,  j\neq k\cr
c_{j,j} &= \int_{\mathbb{R}^n}G(x)(x_j^2-1)\phi_m(x)dx, \quad     {\rm for} \,\,  j= k.
\end{cases}
\end{equation}
We have the following lemma which gives an expression for the asymptotic variance of the second chaos component.

\begin{lemma} \label{lem3.5}
Let $G$ be of Hermite rank $2$ and assume {\bf (C1)}.   Let $C$ be the matrix defined in (\ref{matrixC}). Then,
\begin{equation}  \label{V2}
\lim_{s\to \infty} V_s^{(2)} = V^{(2)} =  \frac 12 \left\|   {\rm Tr} [rC rC]  \right\|_{L^1(\R^n)}.
\end{equation}
\end{lemma}

Suppose in addition that  for every $1 \leq j,k \leq m$, $r_{j,k} \in L^1(\R^n)$. Note that due to the stationarity and mean-square continuity of the fields $\xi_j$s, we have, by Bochner's theorem (Theorem 5.4.1 of \cite{adler_random_2007} or equation 1.2.1 of \cite{ivanov_statistical_1989}), that there exist finite measures $\nu_j$s (called the spectral measures) such that 
\begin{equation}
r_{j,j}(x) = \int_{\mathbb{R}^n} e^{i\langle t,x \rangle} \nu_j(dt).
\end{equation}
Moreover,  due to the  integrability of the covariances, we have that  the $\nu_j$s are absolutely continuous with respect to the Lebesgue measure and admit densities (called spectral densities). Denote the spectral density of $\xi_j$ as $f_j$ and $\alpha_j = \sqrt{f_j}$. 
Set  $\alpha(x) = (\alpha_i(x))_{1\leq i \leq m}$ and let $H(x) = \alpha^T(-x)C\alpha(x)$.  Under these conditions, equation  (\ref{V2}) can be written as
\begin{equation}  \label{3.5}
 V^{(2)} = \frac {(2\pi)^{-n}}2 \| H\| ^2_{L^2(\mathbb{R}^n)}. 
\end{equation}
This formula  has been  motivated by the result obtained in \cite{nicolaescu_clt_2015} in the   context of  the Central Limit Theorem for  the number of critical points, where   $V^{(2)}$ is obtained as the $L^2$-norm of a function.

\section{Preliminaries and chaos expansions}

In this section, we recall the Malliavin operators associated with an isonormal Gaussian process and  the properties of the multiple Wiener-It\^o integrals. We refer the reader to \cite{nualart_malliavin_2006} for a detailed account on this topic. We then write the chaos expansions of  the variables $L_s $ introduced in (\ref{Ls}).

We claim that there exist a Hilbert space $\HH$ and elements $\beta_{j,x}\in \HH$, $1\le i,j \le m$, $x\in \R^n$, such that
\[
r_{i,j} (x-y) = \langle \beta_{i,x} ,\beta_{j,y} \rangle_{\HH}
\]
for all $x,y \in \R^n$ and $1\le i,j \le m$. Indeed, it suffices to choose as $\HH$ the Gaussian subspace of $L^2(\Omega)$ generated by the random field $\xi$ and take $\beta_{i,x} = \xi_i(x)$. 
Consider   an isonormal Gaussian process $X$ on   $\HH$.   That is, $X=\{X(h): h\in \HH\}$ is a Gaussian centered family of random variables,
defined in a probability space $(\Omega, \mathcal{F}, P)$,  such that  $\E[X(h) X(g)] =\langle h,g \rangle_{\HH}$ for any $g,h \in \HH$.   In this situation, $\{\xi_i(x): x \in \R^n, 1\le i\le m\}$ has the same law as $\{X(\beta_{i,x}) :
x\in \R^n , 1\le i\le m\}$.
Therefore, without loss of generality we  can assume the  existence of an isonormal process $X$ on $\HH$ such that 
\begin{equation}   \label{ecu5}
\xi_j(x) = X(\beta_{j,x}).
\end{equation} 
We will also assume that the $\sigma$ field $\mathcal{F}$ is generated by $\xi$.

For a smooth and cylindrical random variable $F= f(X(\varphi_1), \dots , X(\varphi_n))$, with $\varphi_i \in \mathfrak{H}$ and $f \in C_b^{\infty}(\mathbb{R}^n)$ ($f$ and its partial derivatives are bounded), we define its Malliavin derivative as the $\mathfrak{H}$-valued random variable given by
\[
 DF = \sum_{i=1}^n \frac{\partial f}{\partial x_i} (X(\varphi_1), \dots, X(\varphi_n))\varphi_i.
\]
By iteration, we can also define the $k$-th derivative $D^k F$ which is an element in the space $L^2(\Omega; \mathfrak{H}^{\otimes k})$. The Sobolev space $\mathbb{D}^{k,p}$ is defined as the closure of the space of smooth and cylindrical random variables with respect to the norm $\|\cdot\|_{k,p}$ defined by 
\[
 \|F\|^p_{k,p} = \mathbb{E}(|F|^p) + \sum_{i=1}^k \mathbb{E}(\|D^i F\|^p_{\mathfrak{H}^{\otimes i}}),
\]
for any natural number $k$ and any real number $ p \geq 1$. For any Hilbert space $H$, we denote by $\mathbb{D}^{k,p}(H)$ the corresponding Sobolev space of $H$-valued random variables.

 We define the divergence operator $\delta$ as the adjoint of the derivative operator $D$. Namely, an element $u \in L^2(\Omega; \mathfrak{H})$ belongs to the domain of $\delta$, denoted by $\dom\, \delta$, if there is a constant $c_u > 0$ depending on $u$ and satisfying 
\[
|\mathbb{E} (\langle DF, u \rangle_{\mathfrak{H}})| \leq c_u \|F\|_{L^2(\Omega)}
\] for any $F \in \mathbb{D}^{1,2}$.  If $u \in \dom \,\delta$, the random variable $\delta(u)$ is defined by the duality relationship 
\[
\mathbb{E}(F\delta(u)) = \mathbb{E} (\langle DF, u \rangle_{\mathfrak{H}}) \, ,
\]
which is valid for all $F \in \mathbb{D}^{1,2}$.  
In a similar way, for each integer $k\ge 2$, we define the iterated divergence operator $\delta^k$ through the duality relationship 
\[
\mathbb{E}(F\delta^k(u)) = \mathbb{E}  \left(\langle D^kF, u \rangle_{\mathfrak{H}^{\otimes k}} \right) \, ,
\]
valid for any $F \in \mathbb{D}^{k,2}$, where $u\in  {\rm Dom}\, \delta^k \subset L^2(\Omega; \mathfrak{H}^{\otimes k})$.
 
 For any $p>1$ and any integer $k\ge 1$, the operator $\delta^k$ is continuous  from $\mathbb{D}^{k,p} (\HH^{\otimes k})$ into  $L^p(\Omega)$, and we have the  inequality (see, for instance, \cite[Proposition 1.5.4]{nualart_malliavin_2006})
\begin{equation} \label{meyer1}
\|\delta ^k(v) \| _{L^p(\Omega)} \le c_p \sum_{j=0}^k  \| D^jv\|_{L^p(\Omega; \HH^{\otimes j})},
\end{equation}
for any $v\in \mathbb{D}^{k,p} (\HH^{\otimes k})$. This inequality is a consequence of Meyer inequalities (from \cite{Meyer}), which  states the equivalence in $L^p(\Omega)$, for any $p>1$, of the operators $D$ and $(-L) ^{1/2}$, where $L$ is the generator of the Ornstein-Uhlenbeck semigroup introduced below.

\medskip
 Let  $\mathfrak{H}^{\otimes q}$ the $q$-th tensor product of the Hilbert space $\HH$ and denote by $\mathfrak{H}^{\odot q}$ as the subset of $\mathfrak{H}^{\otimes q}$ consisting of all symmetric tensors.  For any $f\in \HH$ we define the  generalized multiple Wiener-It\^o stochastic integral of the symmetric tensor $f^{\otimes q}$ by
 \begin{equation}
 I_q(f^{\otimes q}) =H_q(X(f)),
\end{equation}
where $H_q(x)$ is the $q$-th Hermite polynomial given by  (\ref{hermite}). It is known that the multiple integral $I_q$ can be extended to $\HH^{\odot q}$ and it has the following properties
\begin{equation} \label{mwiprop}
 \mathbb{E}[I_q(f)] = 0,  \qquad \mathbb{E}[I_p(f)I_q(g)] = \delta_p^q \; q! \langle f, g\rangle_{\mathfrak{H}^{\otimes q}}
\end{equation}
 for $f, g \in \mathfrak{H}^{\odot q}$.
 That is,  $I_q $ is a linear isometry  between $ \mathfrak{H}^{\odot q}$ equipped with the modified norm  $\sqrt{q!} \| \cdot \|_{\HH^{\otimes q}}$ and the $q$-th Wiener chaos $\mathcal{H}_q$.
   For an element  $f\in \HH^{\otimes q}$ which is not necessarily symmetric, we define
$I_q(f)= I_q(\widetilde{f})$, where $ \widetilde{f}$ denotes the symmetrization of $f$.

Any element $F\in L^2(\Omega)$ admits a Wiener chaos expansion
\begin{equation} \label{chaos}
F = \sum_{q=0}^\infty I_q(f_q),
\end{equation}
 where  $f_0 = \E[F]$, $I_0$ is the identity on $\R$ and the kernels $f_q\in \HH^{\otimes q}$ are uniquely determined by $F$.

Following appendix B and equation B.4.4 of \cite{nourdin_normal_2012}, we can define the contractions of two   tensors as follows. For two tensors $f = \sum_{j_1,\dots,j_p = 1}^{\infty} a_{j_1,\dots,j_p} e_{j_1} \otimes \cdots \otimes e_{j_p} \in \mathfrak{H}^{\otimes p}$ and $g = \sum_{k_1,\dots,k_q = 1}^{\infty} b_{k_1,\dots,k_q} = e_{k_1} \otimes \cdots \otimes e_{k_q} \in \mathfrak{H}^{\otimes q}$, the $l$-th contraction of $f$ and $g$ ($l \le \min(p,q$))  is the element of $\mathfrak{H}^{\otimes p+q-2l}$ given by 
\begin{equation} \label{contrac1}
f \otimes_l g = \sum_{j_1,\dots,j_p = 1}^{\infty} \sum_{k_1,\dots,k_q = 1}^{\infty} a_{j_1,\dots ,j_p} b_{k_1,\dots,k_q} \prod_{i=1}^l \langle e_{j_i}, e_{k_i} \rangle_{\mathfrak{H}} e_{j_{l+1}} \otimes \cdots \otimes e_{j_p} \otimes e_{k_{l+1}} \otimes \cdots \otimes e_{k_{q}}.
\end{equation}
Notice that even  if $f$ and $g$  are symmetric, the contraction  $f \otimes_l g $ is not necessarily a symmetric tensor.
 Using contractions, we can state the following product formula for multiple Wiener-It\^o integrals.
\begin{equation}  \label{product}
I_p(f)I_q(g) = \sum_{l=0}^{p \wedge q} l! {p \choose l} {q \choose l} I_{p+q-2l}(f \widetilde{\otimes_l} g),
\end{equation}
where $f\in   \HH^{\odot p}$ and  $g\in  \HH^{\odot q}$.

The Ornstein-Uhlenbeck semigroup $\{P_t: t \geq 0\}$ is the  semigroup of operators on $L^2(\Omega)$ defined by
\[
P_t F = \sum_{q=0}^\infty e^{-qt} I_q(f_q),
\]
if $F$ admits the Wiener chaos expansion  (\ref{chaos}). Denote by $L = \frac{d}{dt}|_{t=0} P_t$ the infinitesimal generator of $\{P_t: t \geq 0\}$ in $L^2(\Omega)$. Then we have $LF=-\sum_{q=1}^\infty q J_q(F)$ for any $F \in {\rm Dom}\, L=\mathbb{D}^{2,2} $ where $J_q(F) = I_q(f_q)$. We define the pseudo-inverse of $L$ as 
\begin{equation}  \label{L-1}
L^{-1} F = -\sum_{q=1}^\infty \frac{1}{q} J_q F.
\end{equation} 
The basic  operators $D$, $\delta$ and $L$ satisfy the relation $ LF =-\delta DF $, for any random variable $F\in\mathbb{D}^{2,2}$.
As a consequence, any centered random variable $F\in L^2(\Omega)$ can be  expressed as a divergence:
\begin{equation}  \label{deltad}
F= \delta (-DL^{-1}F).
\end{equation}

\medskip
We now turn to giving the chaos expansions for $L_s$ given by (\ref{Ls}).
For $\beta_{j,x}$ as introduced  in (\ref{ecu5}), we have that for any $x \in \mathbb{R}^n$ and $j \neq k$, under \textbf{(C1)}, 
\begin{equation}  \label{betaj}
\langle\beta_{j,x}, \beta_{k,x}\rangle _{\HH}= \mathbb{E}[\xi_j(x)\xi_k(x)] = 0.
\end{equation}
 Now consider any multi-index $a$ such that $|a| = q$. By the previous facts  (\ref{hermite2}),  (\ref{ecu5})   and taking into account the
 product formula (\ref{product}) and (\ref{betaj}), we can write
\[
\overbar{H}_a(\xi(x)) = \prod_{j=1}^m I_{a_j}(\beta_{j,x}^{\otimes a_j} ) = I_q(\beta_{1,x}^{\otimes a_1} \otimes \cdots \otimes \beta_{m,x}^{\otimes a_m})
\]
We introduce the elements $\rho_x^q$ and $\chi_s^q$ which characterize the expansions. Let 
\begin{equation} \label{ecu2}
\rho_x^q = \sum_{a\in\mathcal{I}_q} c(G, a) \beta_{1,x}^{\otimes a_1} \otimes \cdots \otimes \beta_{m,x}^{\otimes a_m}.
\end{equation}
Notice that, although for each $a\in\mathcal{I}_q$, the tensor $\beta_{1,x}^{\otimes a_1} \otimes \cdots \otimes \beta_{m,x}^{\otimes a_m}$ is not 
necessarily symmetric, the element  $\rho_x^q$ is symmetric because $c(G,a)$ is a symmetric function of the multiindex $a$.
Set
 \begin{equation}  \label{ecu3}
 \chi_s^q = \frac{1}{(2s)^{n/2}} \int_{[-s,s]^n} \rho_x^q dx.
 \end{equation}
  By linearity of the multiple Wiener-It\^o integral and Fubini's theorem for multiple Wiener-It\^o integral, we have that
\[
G_q(\xi(x)) = I_q(\rho_x^q); \; \; \; \; L_s^{(q)} = I_q(\chi_s^q),
\]
where $G_q$ and  $L^{(q)}_s$ are defined in  (\ref{Gq}) and (\ref{Lsq}), respectively.
Therefore, we have the chaos expansion  
\begin{equation} \label{ecu4}
L_s = \sum_{q=d}^{\infty} I_q(\chi_s^q)=\sum_{q=d}^{\infty} L_s^{(q)}.
\end{equation}
This is true because,
\begin{align*}
& \mathbb{E}\Big[\Big(L_s - \sum_{q=d}^l L_s^{(q)}\Big)^2\Big] = \frac{1}{(2s)^n} \mathbb{E}\Bigg[\Big(\int_{[-s,s]^n} G(\xi(x)) - \sum_{q=d}^l G_q(\xi(x))dx\Big)^2\Bigg] \\
 & \qquad  = \frac{1}{(2s)^n} \int_{[-s,s]^n} \int_{[-s,s]^n} \mathbb{E} \Big[\Big(G(\xi(x)) - \sum_{q=d}^l G_q(\xi(x))\Big)\Big(G(\xi(y)) - \sum_{q=d}^l G_q(\xi(y))\Big)\Big] dxdy \\
 & \qquad \leq \frac{\mathbb{E}\big[(G(\xi(0)) - \sum_{q=d}^l G_q(\xi(0)))^2\big]}{(2s)^n} \int_{[-s,s]^n} \int_{[-s,s]^n} dxdy \to 0
\end{align*}
as $l \to \infty$. The last step follows from   stationarity and Cauchy-Schwarz inequality.

\begin{remark}
Due to properties of the multiple Wiener-It\^o integrals noted in \eqref{mwiprop}, we have $\mathbb{E}[L_s^q] = \mathbb{E}[I_q(\chi_s^q)] = 0$ and $\mathbb{E}[G_q(\xi(x))] = \mathbb{E}[I_q(\rho_x^q)]= 0$. Also $\mathbb{E}[G_{q_1}(\xi(x))G_{q_2}(\xi(y))] = 0$ for all $q_1 \neq q_2$.
\end{remark}

\section{Proofs}

\subsection{Proof of Lemma 3.1}

Let us first prove the weak stationarity of the random field $G \circ \xi$.   Taking into account that  $G_q(\xi(x))$ is the projection on the $q$th Wiener chaos of  $G(\xi(x))$, we can write, for any $x,y \in \R^n$,
\[
\mathbb{E}[G(\xi(x))G(\xi(y))] = \sum_{q=d}^{\infty} \mathbb{E}[G_q(\xi(x))G_q(\xi(y))].
\]
Furthermore, in view of the Diagram formula (see \cite{arcones_limit_1994}) we have that $C_{G_q}(x,y)$ depends on the covariances $r_{i,j}(x-y)$ and hence $C_{G_q}(x,y)$ is a function of $x-y$. As a consequence, we get that $C_G(x,y) = C_G(x-y)$ is a function of $x-y$. 

\medskip
To show   (\ref{ecu1a}) we will make use of Lemma 1 of \cite{arcones_limit_1994} and condition \textbf{(C1)}.
We have
\[
(2s)^n  V_s  = \mathbb{E}\Big[\Big(\int_{[-s,s]^n} G(\xi(x)) dx\Big)^2\Big] = \int_{[-s,s]^n} \int_{[-s,s]^n} C_G(x-y) dxdy.
\]
Since by Cauchy-Schwarz inequality and stationarity, $\mathbb{E}[G(\xi(x))G(\xi(y))] \leq \mathbb{E}[(G(\xi(0)))^2]$, we have $V_s < \infty$ for all $s>0$ and
\begin{equation} \label{Vs}
\begin{split}
V_s &= \frac{1}{(2s)^n}\int_{[-s,s]^n} \int_{[-s,s]^n} C_G(x-y) dxdy\\
 			 &= \int_{[-2s,2s]^n} C_G(x) \prod_{i=1}^n \Big(1 - \frac{|x_i|}{2s}\Big) dx\\
 			 &= \int_{\mathbb{R}^n} C_G(x) \prod_{i=1}^n \Big(1 - \frac{|x_i|}{2s}\Big) \textbf{1}_{[-2s,2s]^n}(x) dx \\
 			 &=: \int_{\mathbb{R}^n} C_G(x) I_{2s}(x) dx.
\end{split}
\end{equation}
We set
\begin{equation} \label{psi}
\psi(x) = \Bigg(\sup_{1 \leq i \leq m} \sum_{j=1}^m |r_{i,j}(x)|\Bigg) \vee \Bigg(\sup_{1 \leq j \leq m} \sum_{i=1}^m |r_{i,j}(x)|\Bigg).
\end{equation}
By Lemma 1 of \cite{arcones_limit_1994}, on the set $\{ x: \psi(x) \le 1\}$, we have
\[
|C_G(x)| = |\mathbb{E} [G(\xi(0)) G(\xi(x))]| \leq \psi^d(x) \norm{G}^2_{L^2(\R^m, \gamma_m)}.
\]
Also $\int_{\mathbb{R}^n} \psi^d(x)dx < \infty$ as $\int_{\mathbb{R}^n} |r_{i,j}(x)|^d dx < \infty$ for all $1 \leq i,j \leq m$. On the other hand, 
on the set $\{ x: \psi(x) >1 \}$ we can write, taking into account that $|C_G(x)| \le  \norm{G}^2_{L^2(\R^m, \gamma_m)}$, 
\begin{align*}
 \int_{\{ \psi(x) >1\}}  |C_G(x)| dx &\le \sum_{i,j=1}^m \int_{\{ | r_{i,j}(x)| >\frac 1m\}}  |C_G(x)| dx \\
 & \le \norm{G}^2_{L^2(\R^m, \gamma_m)}  m^d
 \sum_{i,j=1}^m \int_{ \R^n}  |r_{i,j}(x)|^d dx<\infty.
\end{align*}
Observe that $|I_{2s}(x)| = |\Pi_{i=1}^n (1 - \frac{|x_i|}{2s}) \; \textbf{1}_{[-2s,2s]^n}| \leq 1$ for all $s>0$ and as $s \to \infty, I_{2s}(x) \to 1$.   Therefore by dominated convergence theorem,
\[
V = \lim_{s\to \infty} V_s = \int_{\mathbb{R}^n} C_G(x)dx < \infty.
\]
\qed

\subsection{Proof of Theorem 3.2}

We will apply Nualart and Hu's criteria for convergence in distribution to a normal variable (Theorem 3 of \cite{hu_renormalized_2005} or Theorem 6.3.1 of \cite{nourdin_normal_2012}). As a consequence, the theorem follows if the following conditions hold,

\hfill

\noindent 1) For every $q \geq d$, $V_s^{(q)} \to V_q < \infty$ as $s \to \infty$. \\
2) $V = \sum_{q=d}^{\infty} V^{(q)} < \infty$. \\
3) For every $q \geq d$ and every $1 \leq b \leq q-1$, $||\chi_s^q \otimes_b \chi_s^q||_{\mathfrak{H}^{\otimes (2q-2b)}} \to 0$ as $s \to \infty$. \\
4) $\sup_{s>0}\sum_{q=l+1}^{\infty} V_s^{(q)} \to 0$ as $l \to \infty$. \\

Here $\chi_s^q$ is given by  \eqref{ecu3}. Conditions 1), 2) hold by Lemma 3.1. For condition 3), by \eqref{ecu3} we have that
\begin{equation}
\chi_s^q \otimes_b \chi_s^q = \frac{1}{(2s)^n} \int_{[-s,s]^n} \int_{[-s,s]^n} \rho_x^q \otimes_b  \rho_y^q dx dy.
\end{equation} 
Denoting for a multi-index $i = (i_1,...,i_q)$, $\zeta_{i,x} = \beta_{i_1,x} \otimes \cdots \otimes \beta_{i_q,x}$, for the desired convergence to hold, we have, by equation \eqref{ecu2}, that it suffices to show that for any multi-indices $i$ and $j$,
\begin{equation}
J_s := \norm{\frac{1}{(2s)^n} \int_{[-s,s]^n} \int_{[-s,s]^n} \zeta_{i,x} \otimes_b \zeta_{j,y} dx dy}^2_{\mathfrak{H}^{\otimes(2q-2b)}} \to 0
\end{equation}
as $s \to \infty$. We have, using (\ref{contrac1}),
\begin{align*}
& J_s   = \frac{1}{(2s)^{2n}} \int_{[-s,s]^{4n}} \Bigg( \prod_{k=1}^b r_{i_k,j_k}(x-y) r_{i_k,j_k}(z-w)  \\
  & \times  \left  \langle   \left(  \otimes _{ \ell =b+1}^q\beta_{i_{\ell},x}    \right)  \otimes \left( \otimes _{ \ell =b+1}^q \beta_{j_{\ell},y}  \right)  , 
  \left( \otimes _{ \ell =b+1}^q\beta_{i_{\ell},z}  \right) \otimes \ \left( \otimes _{ \ell =b+1}^q \beta_{i_\ell,w}  \right)  \right\rangle_{\HH^{\otimes(2q-2b)}} \Bigg) dxdydzdw. 
\end{align*}
In the above expression, pairing together $\beta_{i_{b+k},x}$ and $\beta_{i_{b+k},z}$ and similarly with the index $j$, we get that,
\begin{align*}
& J_s = \frac{1}{(2s)^{2n}} \int_{[-s,s]^{4n}} \Bigg( \prod_{k=1}^b r_{i_k,j_k}(x-y) r_{i_k,j_k}(z-w) \prod_{k=b+1}^q r_{i_k,i_k}(x-z) r_{j_k,j_k}(y-w) \Bigg) dxdydzdw \\
  & \leq \frac{1}{(2s)^{2n}} \int_{[-s,s]^{4n}} \psi^b(x-y) \psi^b(z-w) \psi^{q-b}(x-z) \psi^{q-b}(y-w) dxdydzdw.
\end{align*}
where $\psi$ is as defined by \eqref{psi}. In what follows,  the value of constant $C$ is immaterial and changes with each step. By H\"older's inequality and the fact that $\psi \in L^q(\R^n)$ for all $q \geq d$ we have that
\begin{equation}
J_s \leq Cs^{-2n} \int_{[-s,s]^{3n}} \psi^b(x-y) \psi^{q-b}(y-w) dxdydw. 
\end{equation}
By the change of variables $(x,y,w) \mapsto (x-y,y-w,w)$ we have
\[J_s \leq Cs^{-n} \int_{[-2s,2s]^{2n}} \psi^b(u)\psi^{q-b}(v)dudv.
\]
We proceed in a manner similar to \cite{darses_limit_2010}. For $k > 0$ denote $T_k = [-k,k]^{2n}$ and $T_k^c$ to be its complement in $\R^{2n}$. Consider
the decomposition
\begin{align*}
 J_s \leq Cs^{-n} \int_{[-2s,2s]^{2n} \cap T_k} \psi^b(u) \psi^{q-b}(v) dudv  + Cs^{-n} \int_{[-2s,2s]^{2n} \cap T_k^c} \psi^b(u) \psi^{q-b}(v) dudv .
\end{align*}
For any fixed $k$, since $\psi$ is bounded, we have that the first term tends to zero as $s \to \infty$. For the second term, by H\"older's inequality we can write
\begin{align*}
& s^{-n} \int_{[-2s,2s]^{2n} \cap T_k^c} \psi^b(u) \psi^{q-b}(v) dudv \\
& \qquad \leq Cs^{-n}\Bigg(s^{n} \int_{\R^n \setminus [-k,k]^n} \psi^q(u)du\Bigg)^{b/q} \Bigg(s^n \int_{\R^n \setminus [-k,k]^n} \psi^q(v)dv\Bigg)^{(q-b)/q}\\
& \qquad \leq C \int_{\R^n \setminus [-k,k]^n} \psi^q(x)dx \to 0
\end{align*}
as $k \to \infty$ yielding the desired conclusion.

Condition 4) also holds as we have, by \eqref{Vs} in Lemma 3.1,
\[
\sum_{q=l+1}^{\infty} V_s^{(q)} =  \sum_{q=l+1}^{\infty} \int_{\mathbb{R}^n} C_{G_q}(x) \; I_{2s}(x)dx 
		                \leq   \sum_{q=l+1}^{\infty} \int_{\mathbb{R}^n} C_{G_q}(x) dx 
		                =   \sum_{q=l+1}^{\infty} V^{(q)} \to 0
\]

as $l \to \infty$ uniformly in $s$.

\qed

\subsection{Proof of Theorem 3.3}

As defined in the statement,
\[
Z_{s,y} = \frac{1}{(2s)^{n/2}} \int_{[-sy^{1/n},sy^{1/n}]^n} G(\xi(x)) dx, \; \; \; y \in [0,\infty).
\]
We gather the necessary notation for the Wiener chaos expansions for the new variables. By the Wiener chaos expansions in \eqref{ecu4}, we have for any $y > 0$, $Z_{s,y} = \sum_{q=d}^{\infty} Z_{s,y}^{(q)} = \sum_{q=d}^{\infty} I_q(\chi_{s,y}^q)$. Here
\[
Z_{s,y}^{(q)} = \frac{1}{(2s)^{n/2}} \int_{[-sy^{1/n},sy^{1/n}]^n} G_q(\xi(x)) dx 
\]
and
\[ \chi_{s,y}^q = \frac{1}{(2s)^{n/2}} \int_{[-sy^{1/n},sy^{1/n}]} \rho_x^q dx.
\]
Due to results by Nualart and Peccati \cite{nualart_central_2005} and Peccati and Tudor \cite{peccati_gaussian_2005} (or see Theorem 2.1 of \cite{campese_continuous_2018}), the convergence of the finite dimensional distributions of $Z_s$ to those of the Brownian motion $\sqrt{V}B_y$ follows if we show that the covariances of the corresponding projections on each Wiener chaos converge. Namely for any $q \geq d$ and $y_1, y_2 > 0$, 
\[
\E[Z_{s,y_1}^{(q)}Z_{s,y_2}^{(q)}] \to V^{(q)} y_1 \wedge y_2
\]
as $s \to \infty$, where $V^{(q)} = \lim_{s \to \infty} V_s^{(q)}.$

Let $y_1 \leq y_2$ and set $s_1 = sy_1^{1/n}$ and $s_2 = sy_2^{1/n}$. Denote $A_s = [-s_1 - s_2, s_1 + s_2]^n$ and $C_s = [s_1 - s_2, s_2 - s_1]^n$. By the change of variables $(x,y) \mapsto (x-y,y)$ we have,
\[
\begin{split}
\mathbb{E}[Z_{s,y_1}^{(q)}Z_{s,y_2}^{(q)}] &= \frac{1}{(2s)^n} \int_{[-s_1,s_1]^n} \int_{[-s_2,s_2]^n} C_{G_q}(x-y) dx dy \\
								  &= \frac{1}{(2s)^n} \int_{C_s} C_{G_q}(u) (2s_1)^n du \\
								  &+ \frac{1}{(2s)^n} \int_{A_s \setminus C_s} C_{G_q}(u) \prod_{i=1}^n \Big(s_1 + s_2 - |u_i|\Big) du. 
								  \end{split}
\]
Due to Lemma 3.1 applied to the random field $G_q(\xi(x))$, we have that as $s \to \infty$
\[
\frac{1}{(2s)^n} \int_{[s_1 - s_2, s_2 - s_1]^n} C_{G_q}(u) (2s_1)^n du \to  V^q y_1
\]
and by dominated convergence theorem the second term converges to zero, that is,
\[
\int_{\mathbb{R}^n} C_{G_q}(u) \prod_{i=1}^n \Big(\frac{y_1^{1/n} + y_2^{1/n}}{2} - \frac{|u_i|}{2s} \Big) \textbf{1}_{[A_s \setminus C_s]} du \to 0.
\]
Therefore, the theorem follows.
\qed

\subsection{Proof of Theorem 3.4}

Since we have established the convergence of the finite dimensional distributions, it now suffices to show that the family of probability measures $\{P_s : s>0\}$ is tight. By problem 4.11 of \cite{karatzas_brownian_2012}, it suffices to show that for some $p > 2$ and for every $T > 0$, the following holds for $0 \leq y_1 \leq y_2 \leq T$,
\begin{equation}
\sup_{s>0} \norm{Z_{s,y_2} - Z_{s,y_1}}_{L^p(\Omega)}\leq C_T |y_2 - y_1|^{1/2}.
\end{equation}
The desired estimate will be obtained by employing a weighted shift operator and obtaining a representation using the divergence operator. We proceed to define the shift operator.

\hfill

If $G\in L^2(\R^m,\gamma_m)$ has rank $d\ge 1$ with the expansion \eqref{Gexp}, for any index $i =1,\dots, m$, we define the   operator $T_{i}$ by
\begin{equation} \label{t1a}
T_{i}(G)(x) = \sum_{q=d}^\infty  \sum_{ a\in \mathcal{I}_q}   c(G,a)  \frac {a_i }qH_{a_i-1} (x_i) \prod _{j=1 , j\not =i} ^m  H_{a_j}(x_j) \,.
\end{equation}
 We know that $G(\xi(x))$ has the Wiener chaos expansion
 \begin{equation}  \label{Gchaos}
 G(\xi(x))= \sum_{q=d}^\infty  \sum_{ a\in \mathcal{I}_q}   c(G,a)  I_q ( \beta_{1,x} ^{\otimes a_1} \otimes \cdots \otimes  \beta_{m,x} ^{\otimes a_m}).
 \end{equation}
 
 The shift operator allows us to represent  $ G(\xi(x))$ as a divergence.
 Notice that this operator is more complicated than the  shift operator considered in the one-dimensional case (see \cite{nourdin_functional_2018}) because we need the weights
 $a_i/q$ in order to have the representation as a divergence. Actually, we are interested in representing $G(\xi(x))$ as an iterated divergence.
 For any $2\le k\le d$ and indexes $i_1, \dots, i_k \in \{1,\dots, m\}$, we can define the iterated operator
\[
T_{i_1. \dots, i_k}= T_{i_1}\circ \stackrel{k)} { \cdots} \circ T_{i_k}.
\]
The following result is our representation theorem.
 
\begin{lemma} \label{lem1} 
 For any $2\le k\le d$, we  have 
\[
 G(\xi(x)) =   \delta^k \left(    \sum_{ i_1, \dots, i_k=1} ^m  T_{i_1, \dots, i_k} G(\xi(x))  \beta_{i_1,x} \otimes \cdots \otimes \beta_{i_k,x} \right) \,.
 \]
\end{lemma}
\begin{proof}
Using the Wiener  chaos expansion (\ref{Gchaos}) and the operator $L^{-1}$ introduced in (\ref{L-1}), we can write
\begin{align*}
L^{-1} G(\xi(x)) &=- \sum_{q=d}^\infty  \sum_{ a\in \mathcal{I}_q}   c(G,a)   \frac 1q   I_q ( \beta_{1,x} ^{\otimes a_1} \otimes \cdots \otimes  \beta_{m,x} ^{\otimes a_m})\\
&=- \sum_{q=d}^\infty    \frac 1q   \sum_{ a\in \mathcal{I}_q}   c(G,a)      \overline{H}_a (\xi(x)).
\end{align*}
This implies, taking into account that $H_m' =mH_{m-1}$, that  
\begin{align}
-DL^{-1} G(\xi(x)) &=  \notag
 \sum_{q=d}^\infty      \sum_{ a\in \mathcal{I}_q}   c(G,a)   \sum_{i=1}^m   \frac {a_i} q H_{a_i-1} (\xi_i(x)) \prod _{j=1 , j\not =i} ^m  H_{a_j}(\xi_j(x)) \beta_{i,x} \\
&=  \sum_{i=1} ^m  T_iG(\xi(x)) \beta_{i,x}.  
\end{align}
Iterating $k$ times this procedure, we can write  
\begin{equation}
(-DL^{-1} )^kG(\xi(x)) =   \sum_{i_1, \dots, i_k=1} ^m  T_{i_1. \dots, i_k} G(\xi(x)) \beta_{i_1,x} \otimes \cdots \otimes \beta_{i_k,x}  .  \label{eq1}
\end{equation}
Taking into account that $-\delta DL^{-1}$ is the identity operator on centered random variables, we obtain 
\begin{align*}
\delta^k(-DL^{-1} )^kG(\xi(x)) &= \delta^{k-1} \delta(-DL^{-1}) [ (-DL^{-1})^{k-1} G(\xi(x))] \\
&=\delta^{k-1}   (-DL^{-1})^{k-1} G(\xi(x)).
\end{align*}
Iterating this relation and using (\ref{eq1}), yields
\[
G(\xi(x))= \delta^k(-DL^{-1} )^kG(\xi(x)) =\delta^k \left(  \sum_{i_1, \dots, i_k=1} ^m  T_{i_1. \dots, i_k} G(\xi(x)) \beta_{i_1,x} \otimes \cdots \otimes \beta_{i_k,x} \right).
\]
Then, the statement in the lemma is a consequence of (\ref{deltad}).
This completes the proof.
\end{proof}

The next  result is the regularization property of the shift operator $T_{i_1, \dots, i_k}$.

\begin{lemma}
Let $p\ge 2$. Suppose that $G\in L^p( \R^m, \phi_m)$. Then  $T_{i_1, \dots, i_k} G(\xi(x))$ belongs to $\mathbb{D}^{k,p}$ for any $k\le d$ and, moreover,
\begin{equation}  \label{ecu1}
\sup_{x\in \R^n} \sup_{1\le i_1, \dots, i_k \le m} \| T_{i_1, \dots, i_k} G(\xi(x)) \| _{k,p} <\infty.
\end{equation}
\end{lemma}

\begin{proof}
  Because $\langle \beta_{i,x} , \beta_{j,x} \rangle_{\HH} = \delta_{ij}$,   using (\ref{eq1}),  we can write for any $x\in \R^n$,
\[
T_{i_1, \dots, i_k} G(\xi(x)) =  \langle (DL^{-1} )^kG(\xi(x)) , \beta_{i_1,x} \otimes \cdots \otimes \beta_{i_k, x } \rangle_{\HH^{\otimes k}}.
\]
Then, by Meyer inequalities, which imply the equivalence in $L^p$ of the operators $D$ and $(-L)^{1/2}$, we can estimate the $\mathbb{D}^{k,p}$-norm of $T_{i_1, \dots, i_k} G(\xi(x)) $ by a constant times the $L^p(\Omega)$-norm of $G(\xi(x))$.
\end{proof}

Let $s_i = sy_i^{1/n}$ and $S_i = [-s_i, s_i]^n$ for i=1,2. We now have
\begin{align*}
&  \|Z_{s,y_2} - Z_{s,y_1} \|_{L^p(\Omega)} \\
& \qquad = \frac 1{ (2s)^{n/2}} \left\| \int_{S_2 \backslash S_1} G(\xi(x)) dx \right\|_{L^p(\Omega)} \\
& \qquad =   \frac 1{ (2s)^{n/2}} \left\| \int_{S_2 \backslash S_1}  
 \delta^d \left(  \sum_{ i_1, \dots, i_d=1} ^m  T_{i_1, \dots, i_d} G(\xi(x))  \beta_{i_1,x} \otimes \cdots \otimes \beta_{i_d,x} \right)  dx \right\|_{L^p(\Omega)} \\
 &\qquad = \frac 1{ (2s)^{n/2}} \left\|  \delta ^d \left(   
  \sum_{ i_1, \dots, i_d=1} ^m  \int_{S_2 \backslash S_1}   T_{i_1, \dots, i_d} G(\xi(x))  \beta_{i_1,x} \otimes \cdots \otimes \beta_{i_d,x} dx  \right)   \right\|_{L^p(\Omega)} .
\end{align*}
Applying Meyer inequalities (see (\ref{meyer1}), we obtain
\begin{align*}
&  \|Z_{s,y_2} - Z_{s,y_1} \|_{L^p(\Omega)} \\
& \qquad  \le C_{p,d}  \sum_{j=0}^d  \frac 1{ (2s)^{n/2}} \left\|   D^j \left(   
  \sum_{ i_1, \dots, i_d=1} ^m  \int_{S_2 \backslash S_1}   T_{i_1, \dots, i_d} G(\xi(x))  \beta_{i_1,x} \otimes \cdots \otimes \beta_{i_d,x} dx  \right)   \right\|_{L^p(\Omega; \HH^{\otimes (j+d)})} \\
  &  \qquad = C_{p,d}  \sum_{j=0}^d  \frac 1{ (2s)^{n/2}} \Bigg( \E  \Bigg|   \sum_{ i_1, \dots, i_d=1} ^m \sum_{ j_1, \dots, j_d=1} ^m  \int_{S_2 \backslash S_1}\int_{S_2 \backslash S_1}
    \langle  D^j  T_{i_1, \dots, i_d} G(\xi(x)) ,   D^j T_{i_1, \dots, i_d} G(\xi(x))  \rangle_{\HH^{\otimes j}} \\
    &  \qquad \qquad  \times  r_{i_1, j_1} (x-y) \cdots r_{i_d, j_d} (x-y) dxdy
    \Bigg | ^{p/2} \Bigg) ^{1/p}.
\end{align*}
Now, using Minkowski's inequality and the estimate  (\ref{ecu1}), we can write
\begin{align*}
&  \|Z_{s,y_2} - Z_{s,y_1} \|_{L^p(\Omega)} \\
& \le   C_{p,d}   \sup_{j=0,\dots, d} \sup_{x\in \R^n} \sup_{i_1,\dots, i_d =1 ,\dots, m}    \|  D^j  T_{i_1, \dots, i_d} G(\xi(x))\| _{L^{p/2}(\Omega; \HH^{\otimes j})}   \\
&\qquad  \times  \frac 1{ (2s)^{n/2}} \left(   \sum_{ i_1, \dots, i_d=1} ^m \sum_{ j_1, \dots, j_d=1} ^m  \int_{S_2 \backslash S_1}\int_{S_2 \backslash S_1}
 |  r_{i_1, j_1} (x-y) \cdots r_{i_d, j_d} (x-y)| dxdy
  \right) ^{1/2} \\
  &\qquad \le C s^{-n/2}   \sum_{i,j=1}^m   \left(      \int_{S_2 \backslash S_1}\int_{S_2 \backslash S_1}
   |  r_{i, j} (x-y)  |^d dxdy \right) ^{1/2} .
\end{align*}
Therefore, we finally obtain
$$
 \|Z_{s,y_2} - Z_{s,y_1} \|_{L^p(\Omega)}  \le C  |y_1-y_2| ^{1/2}  \sum_{i,j=1}^n \left( \int_{\R^n} | r_{i,j} (x) | ^d dx  \right)^{1/2}.
$$
\qed

\subsection{Proof of Lemma 3.5}
Recall that   $C = (c_{j,k})_{1\leq j,k \leq m}$ is the matrix given by (\ref{matrixC}).
For any $j\neq k$, we denote by $a^{(j,k)}$ the multiindex in $\mathcal{I}_2$ such that $a^{(j,k)}_j=1$,  $a^{(j,k)}_k=1$ and  $a^{(j,k)}_\ell=0$ for any $\ell \neq j,k$.  Also $a^{(j,j)}$ will denote the multiindex  in $\mathcal{I}_2$  such that $a^{(j,j)}_j=2$ and  $a^{(j,j)}_\ell=0$ for any $\ell \neq j$. Then,
\[
\mathcal{I}_2 = \{ a^{(j,k)}, 1\le j,k \le m\}.
\]
Moreover, from the definition of the matrix $C$, it follows that for any $j,k=1,\dots,m$,  $j\neq k$,  
\[
c(G, a^{(j,k)} )= c_{j,k}
\]
and for all $j=1,\dots, m$, $ c(G, a^{(j,j)} )=\frac 12 c_{j,j}$.
With this notation we can write
\begin{align*}
 V^{(2)} &= \int_{\mathbb{R}^n} \mathbb{E}[G_2(\xi(0))G_2(\xi(x))] dx \\
 &= \frac 14 \int_{\mathbb{R}^n} \sum_{i,j,k,\ell =1}^m c_{i,j} c_{\ell,k}  \E\left[ \overline{H}_{a^{(i,j)}} (\xi(0))   \overline{H}_{a^{(\ell,k)}} (\xi(x))  \right]dx.
\end{align*}

\medskip
The computation of the expectations $ \E\left[ \overline{H}_{a^{(i,j)}} (\xi(0))   \overline{H}_{a^{(\ell,k)}} (\xi(x))  \right]$ depends on the indexes $i,j, \ell,k$.
Consider the following cases:

\medskip
\noindent
{\it (i) Case $i\neq j$ and $\ell \neq k$}: In this case, we have
\begin{align*}
 \E\left[ \overline{H}_{a^{(i,j)}} (\xi(0))   \overline{H}_{a^{(\ell,k)}} (\xi(x))  \right]&=
  \E\left[  \xi_i(0) \xi_j(0) \xi_{\ell}(x) \xi_k(x)  \right] \\
  &= r_{i,\ell} (x) r_{j,k}(x) + r_{i,k}(x) r_{ j,\ell}(x).
  \end{align*}

\medskip
\noindent
{\it (ii) Case $i\neq j$ and $\ell = k$}: In this case, we have
\begin{align*}
 \E\left[ \overline{H}_{a^{(i,j)}} (\xi(0))   \overline{H}_{a^{(\ell,\ell)}} (\xi(x))  \right]&=
  \E\left[  \xi_i(0) \xi_j(0) (\xi^2_{\ell}(x) -1)  \right] \\
  &= 2r_{i,\ell} (x) r_{j,\ell}(x) .
  \end{align*}

\medskip
\noindent
{\it (iii) Case $i=j$ and $\ell = k$}: In this case, we have
\begin{align*}
 \E\left[ \overline{H}_{a^{(i,i)}} (\xi(0))   \overline{H}_{a^{(\ell,\ell)}} (\xi(x))  \right]&=
  \E\left[ ( \xi^2_i(0) -1)^2 (\xi^2_{\ell}(x) -1)  \right] \\
  &= 2r_{i,\ell} (x) ^2 .
  \end{align*}
  As a consequence, taking into account the symmetry of the matrix $C$, we obtain
  \[
  V^{(2)} 
 =\frac 12 \int_{\mathbb{R}^n} \sum_{i,j,k,\ell =1}^m c_{i,j} c_{\ell,k}  r_{i,\ell} (x) r_{j,k}(x) dx   = \frac 12  \int_{\mathbb{R}^n}  {\rm Tr} [r(x)C r(x)C]dx.
\]
This completes the proof of  Lemma \ref{lem3.5}. 
\qed

\medskip
Finally, we will show  formula  (\ref{3.5}), assuming that the covariances are integrable. To do this, it is convenient to choose a different underlying isonormal Gaussian process.
Let $W$ denote a complex Brownian measure on $\R^n$ and define the isonormal process $X$ on $L^2(\R^n)$ by
\begin{equation} \label{iso1}
X(f) = \int_{\R^n} \mathcal{F}[f](t) W(dt),
\end{equation}
where  $\mathcal{F}[f]$ denotes the Fourier transform of $f \in L^2(\R^n)$.  Recall the following properties of the Fourier transform:
\begin{equation} \label{fourierprop}
\int_{\mathbb{R}^n} f(t) \mathcal{F}[g](t)dt = \int_{\mathbb{R}^n} \mathcal{F}[f](t) g(t) dt,
\end{equation}
for $f,g \in L^2(\R^n)$ and   $\mathcal{F}[\mathcal{F}[f]](x) = (2\pi)^{-n} f(-x)$.

Due to the assumption that for $1 \leq i,j \leq m$, $r_{i,j} \in L^1(\R^n)$, we have that the spectral measures $\nu_j$s of $\xi_j$s are absolutely continuous with respect to the Lebesgue measure and hence $\xi_j$s admit spectral densities. Denoting the spectral density of $\xi_j$ as $f_j$, we have that the following representation holds (see equation 1.2.16 of \cite{ivanov_statistical_1989}).
\begin{equation}
\xi_j(x) = \int_{\mathbb{R}^n} \mathcal{F}[\alpha_j](t - x) \; dW(t),
\end{equation}
where $\alpha_j \in L^2(\mathbb{R}^n)$ are such that $|\alpha_j(t)|^2 = f_j(t)$. Denoting $\beta'_{j,x}(t) = e^{i\langle t,x\rangle} \alpha_j(t)$, we get that $\xi_j(x) = X(\beta'_{j,x})$ and so we have an "embedding" of the field into the isonormal process \eqref{iso1}. 
 Moreover, we have 
\begin{equation} \label{covrel}
r_{j,k}(x) = \mathbb{E}[\xi_j(x)\xi_k(0)] = \langle\beta'_{j,x}, \beta'_{k,0}\rangle_{L^2(\mathbb{R}^n)} = \mathcal{F}[\alpha_j \overbar{\alpha_k}](x).
\end{equation}
As a consequence, we can write
\begin{align*}
  V^{(2)} 
 &=\frac 12 \int_{\mathbb{R}^n} \sum_{i,j,k,\ell =1}^m c_{i,j} c_{\ell,k}  r_{i,\ell} (x) r_{j,k}(x) dx  \\
 &=
\frac 12 \int_{\mathbb{R}^n} \sum_{i,j,k,\ell =1}^m c_{i,j} c_{\ell,k}  \mathcal{F}[\alpha_i \overbar{\alpha_\ell}](x)\mathcal{F}[\alpha_j \overbar{\alpha_k}](x)dx.
\end{align*}
By  Plancherel's theorem, 
\begin{align*}
\int_{\mathbb{R}^n} \mathcal{F}[\alpha_i \overbar{\alpha_\ell}](x)\mathcal{F}[\alpha_j \overbar{\alpha_k}](x)dx &=
\int_{\mathbb{R}^n} \mathcal{F}[\alpha_i \overbar{\alpha_\ell}](x) \overbar{\mathcal{F}[(\alpha_j \overbar{\alpha_k})\circ {\rm sign}]}(x)  dx \\
&=(2\pi)^{-n}\int_{\mathbb{R}^n} \alpha_i (x)\alpha_\ell(x) \alpha_j (-x)\alpha_k(-x)dx,
\end{align*}
where ${\rm sign}(x)=-x$. This
  implies
\[
 V^{(2)} 
= \frac  {(2\pi)^{-n}} 2 ||H||^2_{L^2(\mathbb{R}^n)},
\]
where
\begin{equation}
H(x) = \sum_{j,k=1}^m c_{j,k} \alpha_j(-x) \alpha_k(x)  = \alpha^T(-x)C\alpha(x).
\end{equation}
This completes the proof of (\ref{3.5}).
\qed

\medskip
 
\textbf{Acknowledgement.} The second author expresses his gratitude to Anindya Goswami for discussions. He acknowledges few discussions with Yogeshwaran D. and Sreekar Vadlamani.
 

\begin{thebibliography}{10}

 

\bibitem{adler_random_2007}
R.~J. Adler and J.~E. Taylor, {\em Random fields and geometry}.
\newblock No.~115 in Springer monographs in mathematics, New York: Springer,
  2007.
  
  
\bibitem{arcones_limit_1994}
M.~A. Arcones, Limit {Theorems} for {Nonlinear} {Functionals} of a
  {Stationary} {Gaussian} {Sequence} of {Vectors},  {\em The Annals of
  Probability}, vol.~22, pp.~2242--2274,   1994.
  
\bibitem{breuer_central_1983}
P.~Breuer and P.~Major, Central limit theorems for non-linear functionals of
  {Gaussian} fields, {\em Journal of Multivariate Analysis}, vol.~13,
  pp.~425--441, 1983.

\bibitem{bogachev_gaussian_1998}
V.~I. Bogachev, {\em Gaussian measures}.
\newblock No.~62, American Mathematical Soc., 1998.

\bibitem{campese_continuous_2018}
S.~Campese, I.~Nourdin, and D.~Nualart, Continuous {Breuer}-{Major} theorem:
  tightness and non-stationarity, {\em arXiv:1807.09740 [math]},  2018.
\newblock arXiv: 1807.09740.

 

\bibitem{darses_limit_2010}
S.~Darses, I.~Nourdin, D.~Nualart, {\em et~al.}, Limit theorems for nonlinear
  functionals of volterra processes via white noise analysis, {\em
  Bernoulli}, vol.~16, no.~4, pp.~1262--1293, 2010.

\bibitem{estrade_central_2016}
A.~Estrade and J.~R. Le\'on, A central limit theorem for the {Euler}
  characteristic of a {Gaussian} excursion set, {\em Annals of Probability},
  vol.~44, no.~6, pp.~3849--3878, 2016.
  
  
\bibitem{hariz_limit_2002}
S.~B. Hariz, Limit {Theorems} for the {Non}-linear {Functional} of
  {Stationary} {Gaussian} {Processes}, {\em Journal of Multivariate
  Analysis}, vol.~80, pp.~191--216,  2002.

 

\bibitem{hu_renormalized_2005}
Y.~Hu, D.~Nualart, {\em et~al.}, Renormalized self-intersection local time
  for fractional brownian motion, {\em The Annals of Probability}, vol.~33,
  no.~3, pp.~948--983, 2005.
  
  
\bibitem{ivanov_statistical_1989}
A.~V. Ivanov and N.~N. Leonenko, {\em Statistical {Analysis} of {Random}
  {Fields}}.
\newblock Dordrecht: Springer Netherlands, 1989.
\newblock OCLC: 851393814.

 \bibitem{karatzas_brownian_2012}
I.~Karatzas and S.~Shreve, {\em Brownian motion and stochastic calculus},
  vol.~113.
\newblock Springer Science \& Business Media, 2012. 
 

  \bibitem{Meyer}
P. A. Meyer, 
Transformations de Riesz pour les lois gaussiennes,
{\em Lecture Notes in Math.}, vol.   1059, pp. 179--193, 1984.
 
 
 
 
\bibitem{nicolaescu_clt_2015}
L.~I. Nicolaescu, A {CLT} concerning critical points of random functions on a
  {Euclidean} space, {\em Stochastic Processes and their Applications},
  vol.~127, no.~10, pp.~3412--3446, 2017.
\newblock arXiv: 1509.06200.

\bibitem{nourdin_functional_2018}
I.~Nourdin and D.~Nualart, The functional breuer-major theorem, {\em arXiv
  preprint arXiv:1808.02378}, 2018.

\bibitem{nourdin_normal_2012}
I.~Nourdin and G.~Peccati, {\em Normal {Approximations} with {Malliavin}
  {Calculus}: {From} {Stein}'s {Method} to {Universality}}.
\newblock Cambridge: Cambridge University Press, 2012.

\bibitem{nualart_malliavin_2006}
D.~Nualart, {\em The {Malliavin} calculus and related topics}.
\newblock Probability and its applications, Berlin ; New York: Springer, 2nd
  ed~ed., 2006.

 
 

\bibitem{nualart_central_2005}
D.~Nualart, G.~Peccati, {\em et~al.}, Central limit theorems for sequences of
  multiple stochastic integrals, {\em The Annals of Probability}, vol.~33,
  no.~1, pp.~177--193, 2005.

\bibitem{peccati_gaussian_2005}
G.~Peccati and C.~A. Tudor, Gaussian {Limits} for {Vector}-valued {Multiple}
  {Stochastic} {Integrals}, in {\em Séminaire de {Probabilités} {XXXVIII}}
  (M. \'Emery, M.~Ledoux, and M.~Yor, eds.), Lecture {Notes} in {Mathematics},
  pp.~247--262, Berlin, Heidelberg: Springer Berlin Heidelberg, 2005.

 

\end{thebibliography}
\bibliographystyle{ieeetr}

\end{document}